\documentclass[11pt]{amsart}
\usepackage[T1]{fontenc}
\usepackage{amssymb, amsthm, amsmath}
\usepackage[active]{srcltx}
\usepackage{dsfont}
\usepackage[dvipdf]{graphicx}%
\usepackage{hyperref}

\DeclareMathOperator{\dcl}{dcl}

 \DeclareMathOperator{\dom}{dom}

\DeclareMathOperator{\tp}{tp}

\DeclareMathOperator{\cl}{Cl}

\newtheorem{introtheorem}{Theorem}
\newtheorem{introprop}{Proposition}

\newtheorem{theorem}{Theorem}[section]

\newtheorem{conjecture}[theorem]{Conjecture}
\newtheorem{corollary}[theorem]{Corollary}

\newtheorem{fact}[theorem]{Fact}
\newtheorem{lemma}[theorem]{Lemma}

\newtheorem{proposition}[theorem]{Proposition}

\theoremstyle{definition}
\newtheorem{definition}[theorem]{Definition}
\newtheorem{example}[theorem]{Example}
\newtheorem{remark}[theorem]{Remark}
\newtheorem{question}[theorem]{Question}

\newcommand{\Rr}{{\mathds{R}}}
\newcommand{\Nn}{{\mathds{N}}}
\newcommand{\Qq}{{\mathds{Q}}}

\newcommand{\Zz}{{\mathbb Z}}

\newcommand{\CL}{{\mathcal L}}

\newcommand{\CN}{{\mathcal N}}

\newcommand{\CR}{{\mathcal R}}
\newcommand{\CM}{{\mathcal M}}

\newcommand{\CC}{{\mathcal C}}

\newcommand{\CD}{{\mathcal D}}

\newcommand{\0}{\emptyset}

\renewcommand{\phi}{\varphi}

\long\def\symbolfootnote[#1]#2{\begingroup%
\def\thefootnote{\fnsymbol{footnote}}\footnote[#1]{#2}\endgroup}

\makeatletter

\def\Ind#1#2{#1\setbox0=\hbox{$#1x$}\kern\wd0\hbox to 0pt{\hss$#1\mid$\hss}
\lower.9\ht0\hbox to 0pt{\hss$#1\smile$\hss}\kern\wd0}

\def\Notind#1#2{#1\setbox0=\hbox{$#1x$}\kern\wd0\hbox to 0pt{\mathchardef
\nn=12854\hss$#1\nn$\kern1.4\wd0\hss}\hbox to
0pt{\hss$#1\mid$\hss}\lower.9\ht0 \hbox to
0pt{\hss$#1\smile$\hss}\kern\wd0}

\def\la{\langle}
\def\ra{\rangle}

\def\Ralg{\Rr_{\mathrm{alg}}}

\makeatother

\title[On  weakly o-minimal non-valuational structures]{On definable Skolem functions in weakly o-minimal non-valuational structures}

 \author[P. Eleftheriou]{Pantelis E. Eleftheriou}

\address{Zukunftskolleg and Department of Mathematics and Statistics, University of Konstanz, Box 216, 78457 Konstanz, Germany}

\email{panteleimon.eleftheriou@uni-konstanz.de}

\thanks{The first author was supported by an Independent Research Grant from the German Research Foundation (DFG) and a Zukunftskolleg Research Fellowship. The second author was partially supported by an Israel Science
  Foundation grant number 1156/10.}
\author[A. HASSON]{Assaf Hasson}
\address{Department of mathematics\\
Ben Gurion University of the Negev\\
Be'er Sheva,\\
Israel} \email{hassonas@math.bgu.ac.il}

\thanks{}
\author[G. Keren]{Gil Keren}
\thanks{}
\address{Department of mathematics\\
Ben Gurion University of the Negev\\
Be'er Sheva,\\
Israel}
 \email{kerengi@cs.bgu.ac.il}

\date{\today}

\begin{document}


\begin{abstract} 
We prove that all known examples of weakly o-minimal non-valuational structures have no definable Skolem functions. We show, however, that such structures eliminate imaginaries up to (definable families of) definable cuts. Along the way we give some new examples of weakly o-minimal non-valuational structures. 

\end{abstract}

\maketitle

\section{Introduction}

A fundamental application of o-minimal cell decomposition is the fact that o-minimal expansions of groups admit definable choice, implying -- among others -- the existence of atomic models, elimination of imaginaries and curve selection. 

In the present paper we study the analogous properties in the context of weakly o-minimal structures. Recall that a structure $\CM=\la M, <, \dots\ra$ is  \emph{weakly o-minimal} if $<$ is a dense linear order and every definable subset of $M$ is a finite union of convex sets. Weakly o-minimal structures were introduced by Cherlin-Dickmann \cite{cd} in order to study the model theoretic properties of real closed rings. They were later also used in Wilkie's proof of the o-minimality of real exponential field \cite{WilkiePfaff}, as well as in van den Dries' study of Hausdorff limits \cite{vdDriesLimits}. Macpherson-Marker-Steinhorn \cite{MacMaSt}, followed-up by Wencel \cite{Wenc1, Wenc2}, began a systematic study of weakly o-minimal groups and fields, revealing many similarities with the o-minimal setting. 

An important dichotomy between valuational structures -- those admitting a definable convex sub-group --  and non-valuational ones arose, supported by good evidence that the latter structures resemble o-minimal structures more closely than what the former do. For example, strong monotonicity and strong cell decomposition theorems were proved for non-valuational structures. As a weakly o-minimal expansion of an ordered group cannot admit a definable choice function for the cosets of a non-trivial proper convex sub-group, our study is immediately restricted to the non-valuational case. 

The first results of the present note unravel a discrepancy  between the weakly o-minimal non-valuational setting and the strictly o-minimal setting:  we show that in all known (to us) weakly o-minimal non-valuational expansions of ordered groups definable Skolem functions do not exist. We conjecture that, in fact, no strictly weakly o-minimal expansion of an ordered group admits definable choice. 

The non-existence of definable Skolem functions is proved, essentially, by contradicting (a possible generalisation of) curve selection in \emph{o-minimal traces} (see Definition \ref{trace} below). On the positive side, however, we prove that, using different techniques, elimination of imaginaries -- to a certain extent -- can still be obtained. 

%

Weakly o-minimal  structures arise naturally as expansions of o-minimal structures by externally definable sets \cite{BaiPoi}. Among those, \emph{dense pairs} give rise to non-valuational structures, motivating the following definition:

\begin{definition}\label{trace}
	A structure $\CM$ is an \emph{o-minimal trace} if there exists  an o-minimal expansion of a group, $\CN_0$, in a language $\CL_0$ such that $M\subsetneq N_0$ is dense in $N_0$, $\CM|\CL_0\prec \CN_0$ and $\CM$ is the structure induced on $M$ from $\CN$.
\end{definition}

In other words, $\CM$ is an o-minimal trace if there exists a dense pair $(\CN_0,\CM_0)$ of o-minimal structures (see \cite{vdDriesDense} for details) and $\CM$ is $\CM_0$ expanded by all $\CN_0$-definable sets. In particular no o-minimal trace is o-minimal. To the best of our knowledge all known examples of weakly o-minimal expansions of ordered groups are o-minimal traces. In particular we prove:


\begin{introprop}\label{expansion}
Let $\CM$ be an o-minimal expansion of an ordered group, $\CM'$ any expansion of $\CM$ by cuts. Then $\CM'$  is non-valuational if and only if it is an o-minimal trace. 
\end{introprop}

Section 3 is dedicated to the construction of an example of a weakly o-minimal non-valuational ordered group that is not an o-minimal trace (or even a reduct of an o-minimal trace). \\

The key to our analysis of definable Skolem functions is: 

\begin{definition}\label{nel}
A weakly o-minimal structure $\CM$  \emph{has no external limits} if for every definable $f:(a, b)\to M$ where $a, b\in M\cup \{\pm \infty\}$, the limits $\lim_{t\to a^+} f(t)$ or $\lim_{t\to a^+} f(t)$ exist in  $M\cup\{\pm\infty\}$.
\end{definition}

We recall that by the previously mentioned \emph{strong monotomicity} (or see \cite[Lemma 1.3]{Wenc1}) the above limits always exists in the Dedekind completion of $M$, justifying the above terminology. Of course, all o-minimal structures have no external limits, but beyond that we only know: 

\begin{introprop}\label{thm1}
O-minimal traces have no external limits. 
\end{introprop}



The connection with definable Skolem functions is established by: 

\begin{introprop}\label{thm3} A weakly o-minimal non-valuational expansion of an ordered group is o-minimal if and only if it has  no external limits and, after naming a non-zero constant, admits definable Skolem functions.
\end{introprop}

From this we conclude that o-minimal traces do not have definable Skolem functions. In fact,  since having no external limits is preserved under elementary equivalence and under the passage to ordered reducts, we obtain a stronger result: 

\begin{introtheorem}\label{main}
 If $\CM$ is elementarily equivalent to an ordered group reduct\footnote{This means that the order and the group structure are preserved in the reduct.} of an o-minimal trace, then $\CM$ has no definable Skolem functions.
\end{introtheorem}


The following statement is appealing and appears to be open.

\begin{conjecture}\label{conj}
No strictly weakly o-minimal non-valuational expansion of an ordered group has  definable Skolem functions.
\end{conjecture}

As mentioned above, this conjecture would imply that no strictly weakly o-minimal structure has definable choice. The above conjecture is tightly connected with such questions as the existence of weakly o-minimal non-valuational structures that are not elementarily equivalent to reducts of o-minimal traces, to a better understanding of the notion of external limits, and to generalisations of the theory of dense pairs. These questions are beyond the scope of the present work. 

Section 4 of our note is dedicated to the study of atomic models and elimination of imaginaries in the weakly o-minimal non-valuational setting. We start with a hands-on proof of: 
\begin{introprop}
	Let $\CM$ be an ordered group reduct of an o-minimal trace. Then $\mathrm{Th}(\CM)$ does not have atomic models over arbitrary sets. 
\end{introprop}
On the positive side, we show: 
\begin{introtheorem}
	Let $\CM$ be a weakly o-minimal non-valuational structure. Then $\CM$ eliminates imaginaries up to $\0$-definable families of definable cuts. 
\end{introtheorem}




We note that while this paper was being revised, an independent work Laskowski-Shaw \cite{LaskShaw} appeared, studying weakly o-minimal expansions of ordered groups by a single cut. They show that the structure has definable Skolem functions if and only if it is valuational. 



$ $\\
\emph{Acknowledgements.} We wish to thank Salma Kuhlmann for hosting Gil Keren in Konstanz in the Winter term 2013/2014, during which time this collaboration began. The visit was funded by the Erasmus Mundus EDEN consortium. We also thank Philipp Hieronymi and Moshe Kamensky for their important feedback.

\section{No definable Skolem functions}\label{sec-results}
The present section is dedicated to proving that ordered group reducts of o-minimal traces have no definable Skolem funcitions. For convenience we will assume, throughout, that all groups are equipped with a constant $1>0$. We start with a few preliminaries. 

As usual, if $\CM$ is an ordered structure, a \emph{cut} in $\CM$ is a partition $M=C\cup D$ where $C<D$. A cut $(C,D)$ is \emph{definable} if $C$ is. As a matter of convention, throughout this paper we will assume that if $(C,D)$ is a cut in $M$ then $C$ has no maximum. With these conventions we  let $\overline M$ be the set of all $\CM$-definable (with parameters) cuts ordered by $(C,D)<(C',D')$ if $C'\cap D\neq 0$. Throughout we will implicitly identify $M$ with a subset of $\overline M$ via $a\mapsto((-\infty, a), [a,\infty))$. 

Following the terminology of \cite{MarOmit}, a cut $(C,D)$ is rational if $D$ has a minimal element, and irrational otherwise. If $\CM$ expands an ordered group a cut $(C,D)$ is \emph{non-valuational} if $\inf\{y-x: x\in C, y\in D\}$ exists (in the sense of the structure $\CM$, of course). As pointed out to us by M. Kamensky, if the group is $p$-divisible for some $p$ then such an infimum, if it exists, must be $0$. As a canonical example, if $G$ is a definable proper, non-trivial, convex subgroup of $M$ then the formula $x>G$ defines a \emph{valuational} cut. Indeed, this is the typical case, as shows the following easy observation:  
\begin{fact}\label{nval}
	Let $\CM$ be a weakly o-minimal expansion of a group. Then $\CM$ is non-valuational if and only if no definable cut is valuational. 
\end{fact}


\subsection{Non-valuational expansions of o-minimal groups}\label{sec-cuts}
As already mentioned, it follows from \cite{BaiPoi} that any expansion of an o-minimal structure by externally definable sets is weakly o-minimal. Since any set of cuts over an o-minimal structure is realised in some elementary extension, any expansion of an o-minimal structure by cuts is weakly o-minimal. 

We start by sharpening Fact \ref{nval}. Towards that end we remind that in \cite[Lemma 2.2]{MarOmit} it is shown that if $\CM$ is o-minimal, and $a$ realizes an irrational cut over $M$ then no  $b\in \CM(a)\setminus \CM$ realizes a rational cut over $\CM$. The following lemma is an analogue for non-valuational cuts. It follows from \cite{Wenc-exp}, but we provide a succinct proof.

\begin{lemma}
	Let $\CM$ be an o-minimal expansion of an ordered group. Let $(C,D)$ be an irrational non-valuational cut over $M$ and $C< a< D$ any realization. Then $\CM(a)$ does not realize any valuational cuts over $M$.
\end{lemma}
\begin{proof}
	Let $b\in \CM(a)$ be any element. Then there exists an $\CM$-definable function $f$ such that $b=f(a)$. By o-minimality and the fact that $a\notin M$ (because $(C,D)$ is an irrational cut) there is an $\CM$-definable interval $I$ containing $a$ such that $f$ is continuous and strictly monotone or constant on $I$. If $f$ is constant on $I$ then $b\in M$ and there is nothing to prove. So we may assume without loss of generality that $f$ is strictly increasing. Restricting $I$, if needed, we may also assume that $I$ is closed and bounded. So $f$ is uniformly continuous on $I$.

By \cite[Lemma 2.2]{MarOmit} the type $p\in S_1(M)$ of a positive infinitesimal element is not realized in $\CM(a)$. It follows, since $(C,D)$ is non-valuational, that for any $c\in C(\CM(a))$, $d\in D(\CM(a))$ and $0<\delta \in \CM(a)$ there are $c<c'<d'<d$ with $c'\in C(\CM)$ and $d'\in D(\CM)$ with $d'-c'<\delta$.  Given any $0<\epsilon \in \CM(a)$ take $0<\epsilon'<\epsilon$ with $\epsilon'\in \CM$ and let $\delta\in M$ be such that $|f(x)-f(y)|<\epsilon'$ for any $x,y\in I$ with $|x-y|<\delta$. So
\[\inf\{f(d)-f(c): c\in C(\CM(a)), d\in D(\CM(a)) \}=0.\]

Thus, $b$ realizes a non-valuational cut over $\CM$. Since $b\in \CM(a)$ is arbitrary, this finishes the proof of the lemma.
\end{proof}

We   also need the following observation.\footnote{We thank Y. Peterzil for pointing out this formulation.}

\begin{proposition}\label{propcuts}
	Let $\CM$ be an o-minimal expansion of an ordered group and $\mathcal C=\{(C_i, D_i)\}_{i\in I}$ a collection of irrational non-valuational cuts in $\CM$. Then there exists $\CM\prec \CN$ such that $M$ is dense in $N$ and $\CN$ realizes all cuts in $\mathcal C$.
\end{proposition}
\begin{proof}
	Let $\{(C_i, D_i)\}_{i\in I}$ be a collection of cuts as in the assumption, and $p\in S_1(M)$ the type of a positive infinitesimal. We construct $\CN$ by induction as follows. For $i=1$ let  $\CM_1:=\CM(a_1)$ where $a_1\models (C_1, D_1)$ is any realization. By \cite[Lemma 2.2]{MarOmit} $p$ is not realized in $\CM_1$. Let $a<b\in M_1$ be any elements. Then there exists $r\in M$ such that $0<r<b-a$ (otherwise $b-a\models p$ ). By the previous lemma $a$ realises a non-valuational cut, so there exists $a'\in M$ such that $a'+\frac{r}{2}>a$. Thus $a'+\frac{r}{2}\in (a,b)\cap M$. Since $a,b$ were arbitrary this shows that $M$ is dense in $M_1$. 
	
	Assume now that for all $j<i$ we have constructed $\CM_j$ such that $(C_j, D_j)$ is realized in $\CM_{j+1}$ and such that $M$ is dense in $M_j$ (so in particular, $\CM_j$ does not realise $p$). If $i$ is a successor ordinal we let $\CM_{i, 0}=\CM_{i-1}$ and if $i$ is limit we let $\CM_{i,0}:=\bigcup_{j<i} \CM_i$.  Note that as density is preserved under passing to the limit,  by induction, $M$ is dense in $M_{i,0}$.  Finally, if $(C_i, D_i)$ is realized in $\CM_{i,0}$ set $\CM_{i}=\CM_{i,0}$. Otherwise set $\CM_i:=\CM_{i,0}(a_i)$ where $a_i\models (C_i, D_i)$ is any realization. 
	
	We prove that $\CM_{i}$ does not realize $p$. If $\CM_i=\CM_{i,0}$ this follows from the induction hypothesis. So we assume that this is not the case. Thus $(C_i(\CM_{i,0}), D_i(\CM_{i,0}))$ defines a cut in $\CM_{i,0}$, and as by induction $M$ is dense in $M_{i,0}$ this cut is still non-valuational. Thus, applying \cite[Lemma 2.2]{MarOmit} again the desired conclusion follows. As in the induction base, it follows that $M$ is dense in $M_i$. 

	Setting $\CN:=\bigcup_{i\in I} \CM_i$, by construction, $\CN$ realizes all $(C_i,D_i)_{i\in I}$ and, by induction $M$ is dense in $N$, as required. 
\end{proof}

From the above result we can deduce the following which, though not needed in this paper, may be of independent interest: 
\begin{corollary}
	Let $\CM$ be an o-minimal expansion of a group. Then there exists a maximal elementary extension $\CM\prec \CN$ such that $M$ is dense in $N$. 
\end{corollary}
\begin{proof}
	Let $(C_i,D_i)_{i\in I}$ enumerate all non-valuational irrational cuts in $M$. Let $\CN$ be the structure provided by the previous proposition with respect to this collection of cuts.  
	
	Let $\CN_1\succ \CN$ be any proper extension and $a\in N_1\setminus N$. Let $p:=\tp(a/N)$. We may assume that $N$ is dense in $N_1$ (otherwise $M$ is certainly not dense in $N_1$ and we have nothing to prove). So $p$ is non-valuational. Since $M$ is dense in $N$ it is also non-valuational in $M$. If $p$ is irrational over $N$ it is irrational over $M$, and so realised in $N$, which is impossible. So $p$ is rational, and there exists $a'\in N$ such that $p$ is an infinitesimal near $a'$. But then, say, $(a'-a)\cap N=\0$, so $N_1$ is not dense in $M$. 
\end{proof}

Returning to our main argument we can now deduce Proposition \ref{expansion}.

\begin{proposition}
	Let $\CM$ be an o-minimal expansion of an ordered group. Then any expansion of $\CM$ by non-valuational cuts is an o-minimal trace.
\end{proposition}
\begin{proof}
	Let $\widetilde \CM$ be the expansion of $\CM$ by unary predicates $\{C_i\}_{i\in I}$ interpreted as distinct irrational non-valuational cuts in $M$. We have to show that there exists an elementary extension $\CM\prec \CN$ such that $M$ is dense in $N$ and $\widetilde \CM$ is precisely the structure induced on $\CM$ by all externally definable subsets from $\CN$.
	
	Let $\CN$ be as in Proposition \ref{propcuts}, realizing all $C_i$. Then $\CM\prec \CN$ and $M$ is dense in $N$, so $(\CN,\CM)$ is a dense pair. By \cite[Theorem 2]{vdDriesDense} the structure induced on $\CM$ in the pair $(\CN,\CM)$ is precisely the expansion of $\CM$ by unary predicates for all cuts realized in $\CN$. Thus, by construction, we get that $\widetilde M$ is a reduct of the structure induced on $\CM$ from $(\CN, \CM)$. So it remains to show that any cut over $M$ definable in $(\CN, \CM)$ is definable in $\widetilde \CM$.
	
So let $a\in N$ be any element. We have to show the $(-\infty, a)\cap M$ is definable in $\widetilde \CM$. By construction there are $a_1,\dots, a_n$ realizing the cuts $C_{i_1}, \dots, C_{i_n}$ and an $\CM$-definable continuous function, $f$, such that $f(a_1,\dots, a_n)=a$. Choose $a_1,\dots, a_n$ and $f$ so that $n$ is minimal possible.
For every $\eta\in \{-1,1\}^{n}$ say that $f$ is of type $\eta$ at a point $\bar c\in N^n$ if
\[f_i(x_i):=f(c_1,\dots, c_{i-1},x_i,c_{i+1},\dots, c_n)\]
 is strictly monotone at $c_i$ and for all $1\le i\le n$ and $f_i(x_i)$ is increasing at $c_i$ if and only if $\eta(i)=1$.
By the minimality of $n$ there is some $\eta\in \{-1,1\}^n$ such that $f$ is of type $\eta$ at $(a_1,\dots, a_n)$. In particular, the set $F_{\eta}$ of points $\bar x$ such that $f$ is of type $\eta$ at $x$ is $\CM$-definable with non-empty interior.  Let $L_i:=C_i$ if $\eta(i)=-1$ and $L_i:=D_i$ (the complement of $C_i$) otherwise, then $L:=\prod_{i=1}^n L_i\cap F_{\eta}$ has non-empty interior and definable in $\widetilde M$, and for any $\bar x\in L$ we have that $f(\bar x)<f(a_1,\dots, a_n)=a$. Since $M$ is dense in $N$ and $L(N)$ is not empty also $L(M)$ is not empty. Thus
\[x\in (-\infty, a)\iff \exists y(y\in L \land x<f(y)) \]
and the right hand side is $\widetilde M$-definable.
\end{proof}

\subsection{No external limits}\label{sec-nel}
The key to proving Proposition \ref{thm3} is the fact that o-minimal traces do not have external limits (Propostion \ref{thm1}). We will need the following fact (which follows from \cite[Lemma 1.3]{HaOn2}):

\begin{fact}\label{HaOn}
	Let $\CM\prec \CN$ be o-minimal structures.
	\begin{enumerate}
		\item If $f: N\to N$ is a definable function such that $f(a)\in M$ for all $a\in M$ then there are finitely many $N$-definable intervals $\{I_i\}_{i=1}^k$ and $M$-definable functions $f_i$ such that $\bigcup_{i=1}^k (I_i\cap M)=M$ and $f|I_i=f_i$.
		\item If $M$ is dense in $N$ and $f:N^n\to N$ is continuous with $f(a)\in M$ for all $a\in M^n$ then $f$ is $M$-definable.
	\end{enumerate}
\end{fact}

We can now prove a slightly stronger statement than Proposition \ref{thm1}: 

\begin{lemma}\label{NoEL}
	Let $\CM\prec \CN$ be o-minimal structures. Let $\CM_1$ be the structure induced on $\CM$ by all $\CN$-definable sets. Then $\CM_1$ has no external limits.
\end{lemma}
\begin{proof}

If $a<b$ are elements in $M$ and $f:(a,b)\to M$ is an $\CN$-definable function then by \cite{BaiPoi} there exists an $\CN$-definable function $F:(a,b)\to N$ such that $F|M=f$. By Fact \ref{HaOn} there is an $\CN$-definable interval $\cl(I)\ni b$ such that $F|I\cap M$ is the graph of an $\CM$-definable function $g$. In particular $\lim_{x\to b}f(x)=\lim_{x\to b}g(x)$. By o-minimality of $\CM$ we know that either $g$ is unbounded near $b$, in which case we have nothing to prove, or $\lim_{x\to b}g(x)\in M$, proving the claim.
\end{proof}

As a special case, we obtain Proposition \ref{thm1}.

\begin{corollary}\label{nEL}
	Any expansion of an o-minimal group by externally definable sets has no external limits. In particular: 
\begin{enumerate}
	\item  Every o-minimal trace has no external limits.
	\item There are weakly o-minimal valuational structures, e.g., RCVF, with no external limits. 
\end{enumerate}
\end{corollary}

\begin{remark}
	Corollary \ref{nEL}(2) answers negatively a question of Peterzil (private communication) who asked whether having no external limits is characteristic of non-valuational structures. 
\end{remark}

There is an explicit connection between definable non-valuational cuts, definable Skolem functions and no external limits: 

\begin{lemma}\label{DC1}
	Let $\CM$ be a weakly o-minimal expansion of an ordered group, $C$ an $\CM$-definable irrational non-valuational cut such that $0\in C$. Then the formula 
	\[\phi_C(x):=\exists y(x>0\land y\in C\land x+y\notin C)\]
	has a definable Skolem function only if $\CM$ has external limits. 
\end{lemma}
\begin{proof}
	Because $C$ is non-valuational, $\CM\models \phi_C(a)$ for all $a>0$. Thus, if $f:(0,\infty)\to M$ is a Skolem function for $\phi_C(x)$ we must have $\lim\limits_{x\to 0}f(x)=\sup(C)$. Because $C$ is irrational, if $f$ is definable it witnesses that $\CM$ has external limits. 
\end{proof}

We have thus proved: 

\begin{proposition}\label{DC}
	Let $\CM$ be a weakly o-minimal non-valuational expansion of an ordered group. Then $\CM$ is o-minimal if and only if it has no external limits and $\CM$ admits definable Skolem functions.
\end{proposition}
\begin{proof}
	The left-to-right direction is well-known. In the other direction, if $\CM$ is not o-minimal it has at least one definable (irrational) non-valuational cut, $C$. Since by assumption $\CM$ has no external limits, by Lemma \ref{DC1} $\CM$ has no definable Skolem functions. 
\end{proof}

In the context of o-minimal traces we can give a more precise statement, Theorem \ref{main} : 

\begin{corollary}\label{DC2}
	If $\CM$ is elementarily equivalent to an ordered group reduct of an o-minimal trace then $\CM$ has definable Skolem functions if and only if it is o-minimal.
\end{corollary}
\begin{proof}
	By Corollary \ref{nEL} o-minimal traces have no external limits. Since both having no external limits and being weakly o-minimal and non-valuational are elementary and preserved under reducts $\CM$ has no external limits, the result follows from Proposition \ref{DC}
\end{proof}

For reducts of o-minimal traces we can show an even stronger result: 
\begin{proposition}\label{reducts}
	Let $\CM$ be an ordered group reduct of an o-minimal trace. Then no expansion of $\CM$ by externally definable sets has definable Skolem functions, unless $\CM$ is o-minimal. . 
\end{proposition}

Before proceeding to the proof of the proposition we need some handle over externally definable sets in the context of (reducts of) o-minimal traces. First, we need some terminology: 

\begin{definition}
	An o-minimal structure $\widetilde \CM$ in a signature $\CL$ witnesses that $\CM$ is an o-minimal trace, if $(\widetilde \CM, \CM|\CL)$ is a dense pair of o-minimal structures and $\CM$ is the structure induced on $M$ from $\widetilde \CM$. 
\end{definition}

%
%

Then: 

\begin{lemma}\label{external}
	Let $\CM$ be an o-minimal trace witnessed by $\widetilde{\CM}$ in the signature $\CL$. Then any expansion of $\CM$ by externally definable sets is definable in the (non-dense) pair $(\widetilde\CN, \CM|\CL)$ where $\widetilde{\CN}\succ \widetilde{\CM}$ is saturated. 
\end{lemma}
\begin{proof}
	Let $(\widetilde{\CN_1}, \widetilde\CN_0)\succ (\widetilde{\CM}, \CM|\CL)$ be saturated. So the induced structure on $N_0$ in the pair is a saturated model of $\mathrm{Th}(\CM)$. Let us denote this structure $\CN$. 
	
	By \cite{BaiPoi} every definable set in $\CN$ is of the form $D\cap N_0^k$ for some $\widetilde{\CN_1}$-definable set $D$. So any externally definable set in $\CM$ is of the form $D\cap M^k$ for some $\widetilde{\CN_1}$-definable set $D$. This is what we needed.
\end{proof}

The above lemma generalises automatically to the case where $\CM$ is an ordered group  reduct of an o-minimal trace. So we are reduced to proving the following: 

\begin{lemma}\label{nSF}
	Let $\CM\prec \CN$ be o-minimal expansions of groups. Assume that there exists $c\in N\setminus M$ realising a non-valuational cut $C$ over $M$. Then no expansion of $\CM$ by $\CN$-definable sets has a definable Skolem function for the formula 
		\[\phi_C(x):=\exists y(x>0\land y\in C\land x+y\notin C)\]
\end{lemma}
\begin{proof}
By Lemma \ref{NoEL} and Proposition \ref{DC1}.
\end{proof}

Combining all the above observations we get a proof of Proposition \ref{reducts}: 
\begin{proof}[Proof of Proposition \ref{reducts}]
	Assume $\CM$ is not o-minimal. Being non-valuational it admits at least one definable non-valuational irrational cut, $C$. By Lemma \ref{external} any externally definable expansion of $\CM$ is a reduct of the structure induced on $M$ in some o-minimal pair $(\CN_0, \CM_0)$. Since the structure induced on $M$ in the pair is an expansion of $\CM$ it admits at least one definable irrational non-valuational cut. So by Lemma \ref{nSF} this expansion has no definable Skolem functions. 
\end{proof}


\begin{remark}
	There are good reasons to believe that Proposition \ref{reducts} extends to structures elementarily equivalent to such reducts. It seems, however, that to prove such a result more sophisticated techniques are required, going beyond the scope of the present note. 
\end{remark}

We point out that the assumption in Proposition \ref{DC} of $\CM$ expanding an o-minimal ordered group is necessary.

\begin{example}
	Let $\CR$ be the structure obtained by appending two real closed fields one ``on top'' of the other. More precisely, the language is given by $(\le,R_1, R_2, +_1, \cdot_1, +2, \cdot_2)$ and the theory of $\CR$ is axiomatised by:
	\begin{enumerate}
	\item $R_1, R_2$ are unary predicates such that $(\forall x)(R_1(x)\leftrightarrow \neg R_2(x))$ and $(\forall x,y)(R_1(x)\land R_2(y)\to x<y)$.
	\item $+_i, \cdot_i$ are ternary relations supported only on triples of elements in $R_i$. They are graphs of functions on their domains, and $R_i$ is a real closed field with respect to these operations.
	\item $\le$ is an order relation compatible with the field ordering of $R_1$ and $R_2$ together with (1) above.
\end{enumerate}
It follows immediately from quantifier elimination for real closed fields, that the above theory is complete and has quantifier elimination (after adding constants for $0,1$ in both fields, and relation symbols for the inverse function in both fields). Thus $\CR$ is weakly o-minimal, and the only definable cut in $\CR$ not realized in $\CR$ is $(R_1(R),R_2(R))$. However, $\CR$ does have external limits. Take the function $x\mapsto x^{-1}$ in the field structure on $R_1$ on the interval $(0,1)$. Clearly its limit, as $x\to 0^+$ is $*$.

As a note, the boolean algebra of definable subsets of $R^n$ (any $n$) is the boolean algebra generated by sets of the form $S_1\times S_2$ where $S_i\subseteq R_i^{n_i}$ are semi-algebraic sets with $n_1+n_2=n$ and closing under the natural action of $\mathrm{Sym}(n)$. It follows that $\CR$ has definable Skolem functions.
\end{example}


\section{A new example}\label{sec-example}

Our initial approach to proving Conjecture \ref{conj} was to verify whether all weakly o-minimal non-valuational structures are o-minimal traces. As it turns out, the class of o-minimal traces is not closed under taking ordered group reducts. We do not prove this here. The present section 
is dedicated to an example, $\Qq_{vs}^\pi$, of a weakly o-minimal expansion of the ordered group of rational numbers which is not a reduct of an o-minimal trace. However,  $\Qq_{vs}^\pi$ is elementarily equivalent to a reduct, $\Ralg^\pi$, of an o-minimal trace. This shows that the class of reducts of o-minimal traces is not elementary. We do not know the answer to the analogous question for o-minimal traces. \\

We first construct $\Ralg^\pi$. Some preliminary work is needed.

\begin{lemma}
	Let $\CN$ be an o-minimal structure, $M\subseteq N$ a dense subset. Assume that $\CN'$ is an o-minimal structure with universe $N$ and the same order relation as $\CN$, and such that for any $\CN$-definable set $S\subseteq N^n$ there exists an $\CN'$-definable set $S'$ such that $S\cap M^n=S'\cap M^n$. Then any $\CN$-definable set is $\CN'$-definable. 
\end{lemma}
\begin{proof}
	By cell decomposition, in $\CN$ it will suffice to show that any $\CN$-definable open cell is also $\CN'$-definable. So let $C$ be such an open cell. Then $C=\mathrm{int}\cl(C)$. In addition, because $C$ is open and $M$ is dense in $N$ we also get $\cl(C\cap M^n)=\cl(C)$.
	
	Let $C'$ be $\CN'$-definable such that $C'\cap M^n=C\cap M^n$. Since $\cl(C)=\cl(C\cap M^n)$ we get that $\cl(C')\supseteq \cl(C)$, implying that $C=\mathrm{int}\cl(C)\subseteq \mathrm{int}\cl(C')$.
	
	On the other hand, if $\0\neq \mathrm{int}\cl(C')\setminus \cl(C)$ then $M^n\cap (\mathrm{int}\cl(C')\setminus \cl(C))\neq \0$. By o-minimality of $\CN'$ this implies that there exists a definable open box $B\subseteq C'$ such that $B\cap \cl(C)=\0$. This contradicts the assumption that $C\cap M^n=C'\cap M^n$.
		
	It follows that $C=\mathrm{int}\cl(C')$. Since the right hand side is $\CN'$-definable and $C$ was an arbitrary open cell the lemma is proved.
\end{proof}

\begin{corollary}\label{unique}
	Let $\CN$, $\CN'$ be o-minimal structures with universe $N$ and the same underlying order. Assume that for some dense $M\subseteq N$, the trace of $\CN$ on $M$ and the trace of $\CN'$ on $M$ are the same. Then $\CN$ and $\CN'$ have the same definable sets.
\end{corollary}

We now recall the following fact from \cite{Wenc1}. Given  a weakly o-minimal non-valuational expansion $\CM$ of an ordered group, let $\overline M$ denote the set of all definable cuts. A function  $f: M^n\to \overline M$ is called \emph{definable} if $\{(x,y): y\ge f(x) \}$ is definable in $\CM$, and it is called \emph{strongly continuous} if it extends continuously to (a necessarily unique)  $\overline f:\overline M^n\to \overline M$.

\begin{fact}\label{fact-Wenc}
There exists an o-minimal structure $\overline \CM$ on $\overline M$ inducing on $M$ precisely the structure $\CM$. Moreover, if $f$ is strongly continuous, its continuous extension $\overline  f$ is definable in $\overline \CM$.
\end{fact} 
Wencel calls $\overline \CM$ the ``canonical o-minimal extension'' of $\CM$. Corollary \ref{unique} and Fact \ref{fact-Wenc} give some justification to this terminology. 

\begin{corollary}\label{unique1}
	Let $\CM$ be a weakly o-minimal non-valuational expansion of an ordered group. Then $\overline \CM$ is the unique -- up to a change of signature -- o-minimal structure on $\overline M$ whose induced structure on $M$ is $\CM$.
\end{corollary} 

We now proceed to define the structure $\Ralg^\pi$. Let $\Rr$ be the field of real numbers, $\Ralg$ the real closure of $\Qq$. Then $(\Rr, \Ralg)$ is a dense pair. Let $\widetilde{\Ralg}$ be the structure on $\Ralg$ induced from $\Rr$, and $\Ralg^{\pi}$ the reduct ${(\Ralg, \le,0,1, +, \pi\cdot)}$, that is, the additive group of the field of real algebraic numbers equipped with  the unary function $x\mapsto \pi x$. The structure $\Ralg^\pi$ is a weakly o-minimal non-valuational expansion of an ordered group with no external limits, being the reduct of $\widetilde \Ralg$, which has these properties by virtue of being an o-minimal trace. Our first goal is to prove a quantifier elimination result for the theory of $\Ralg^\pi$ (Proposition \ref{QE} below).

\begin{lemma}\label{ralg}
	For all $\alpha\in \Qq(\pi)$ the relation $\alpha x<y$ is $\0$-definable in $\Ralg^\pi$. Moreover, for $\alpha_1,\dots, \alpha_n\in \Qq(\pi)$ linearly independent over $\Qq$ \\
	the relation $\sum_{i=1}^n \alpha_i x_i <0$ is $\0$-definable.
\end{lemma}
\begin{proof}
	Abusing terminology, we will say that $\alpha\in \Qq(\pi)$ is definable in $\Ralg^\pi$ if the relation $\alpha x <y$ is.  We  show that if $\alpha, \beta\in \Qq(\pi)$ are definable in $\Ralg^\pi$ then so is $\alpha+\beta$ and that $\pi^{n}$ is definable for all $n\in \Zz$. 	Indeed, if $\alpha, \beta$ are definable by $S_\alpha, S_\beta$ then $\alpha+\beta$ is defined by
	\[(\exists z_1,z_2)(S_\alpha(x,z_1)\land S_\beta(x,z_2)\land y>z_1+z_2).\]
	If $x\mapsto \pi^n x$ is definable by $P_n(x,y)$ then $x\mapsto \pi^{n+1}x$ is defined by
	\[
	(\exists z)(P_n(x,z)\land P_1(z,y).
	\]
	So to conclude the first part of the lemma it remains only to note that $P_{-1}(x,y)$ is given by $P_1(y,x)$.
	
	For the second part of the lemma we will show that if $n>1$  and  $\alpha_1, \dots, \alpha_n$ are linearly independent over $\Qq$ then:
	\[\sum_{i=1}^n \alpha_i x_i < 0 \iff (\forall x_1',\dots x_n')\left (\bigwedge_{i=1}^n x_i'<\alpha_i x \to \sum_{i=1^n} x_i'< 0 \right )
	\]	
	unless $x_i=0$ for all $i$.
	The left-to-right direction is clear, so we have to show the other implication. The assumption implies $\sum_{i=1}^n \alpha_i x_i  \le 0$, so we only have to check that equality cannot hold. First, observe that since the $x_i$ are not all $0$, we may assume -- by induction on $n$ -- that the $x_i$ are linearly independent over $\Qq$. Indeed, if $x_1=\sum_{i=2}^n q_i x_i$ for $q_i\in \Qq$ we get
	\[
	\sum_{i=1}^n \alpha_i x_i=\alpha_1\sum_{i=2}^n q_i x_i +\sum_{i=2}^n \alpha_i x_i=\sum_{n=2}^n (q_i\alpha_1+\alpha_i)x_i
	\]
	and as $\{(q_i\alpha_1+\alpha_i)\}_{i=2}^n$ are still independent over $\Qq$ the claim follows. 
	Now $\alpha  _i$ are polynomials in $\pi$ with rational co-efficients and $x_i$ are real algebraic numbers, so we can write $\sum_{i=1}^n \alpha_i x_i=\sum_{i=0}^k \beta_i \pi^i=0$ where  $\beta_i$ are $\Qq$-linear combinations of $x_1,\dots, x_n$. So $\beta_i=0$ for all $i$ if and only if the $\alpha_i$ are all $0$, which is impossible, since they are linearly independent.
\end{proof}

The above lemma can be restated as follows: 
\begin{corollary}
	Let $\CL^\pi$ be the language of ordered $\Qq$-vector spaces expanded by $n$-ary predicates for the relations $C_{\bar \alpha}(\bar x):=\sum_{i=1}^n \alpha_ix_i<0$ for all $n\in \Nn$ and $\alpha_i\in \Qq(\pi)$. Then $C_{\bar \alpha}(\bar x)$ is definable in $\Ralg^\pi$  and $C_{\bar \alpha}(\bar r)\equiv C_{\bar \beta}(\bar s)$ if and only if $\sum \alpha_i\otimes r_i=\sum \beta_j \otimes s_j$ as elements of $\Rr\otimes_\Qq \Qq(\pi)$. 
\end{corollary}


Let $T_0$ be the $\CL^\pi$-theory of ordered $\Qq$-vector spaces expressing the conclusion of the previous corollary. By construction $\Ralg^\pi\models T_0$. 

\begin{proposition}\label{QE}
	The theory $T_0$ has quantifier elimination. 
\end{proposition}
\begin{proof}
	Let $\mathcal Q_1, \mathcal Q_2\models T_0$ be saturated of the same cardinality. We will show that if $A_i\subseteq \mathcal Q_i$ are small divisible subgroups, and $f: A_1\to A_2$ is a partial $\CL^\pi$-isomorphism, then $f$ can be extended to any $a\in \mathcal Q_1$.
	
	Since $f$ is an $\CL^\pi$-isomorphism $\mathcal Q_i\models T_0$ we can extend $f$ to an isomorphism of $A_1\otimes \Qq(\pi)$ with $A_2\otimes \Qq(\pi)$. Identify $a$ with $a\otimes 1$. By quantifier elimination in the theory of ordered divisible abeian groups and since $A_1\otimes \Qq(\pi)$ is a divisible abelian sub-group, $\tp(a/A_1\otimes \Qq(\pi))$ is determined by the cut it realises. Thus, it will suffice to show that the same cut over $A_2\otimes \Qq(\pi)$ is realised in $\mathcal Q_2\otimes \Qq(\pi)$. But since $\mathcal Q_2$ is a saturated model of $T_0$ we automatically get that $\mathcal Q_2\otimes \Qq(\pi)$ is saturated (for $\le)$, as required.
\end{proof}


We are now ready to present our main example. Let 
$$\Qq_{vs}^\pi =(\Qq, \le, 0,1, +, \pi\cdot).$$
Clearly, $\Qq_{vs}^\pi \models T_0$, so by quantifier elimination (Proposition \ref{QE}) $\Qq_{vs}^\pi \equiv \Ralg^\pi$. In particular, since $\Ralg^\pi$ is weakly o-minimal and non-valuational, so is $\Qq_{vs}^\pi$.

\begin{theorem}
	$\Qq_{vs}^\pi$ is not a reduct of an o-minimal trace.
\end{theorem}
\begin{proof}
	Assume towards a contradiction that there exists an o-minimal structure $\mathcal Q$ with universe $\Qq$, and $\CR\succ \mathcal Q$ such that $(\CR, \mathcal Q)$ is a dense pair, and the structure induced on $\mathcal Q$ from $\CR$ expands $\Qq_{vs}^\pi$.
	
	The desired conclusion now follows from \cite{MilStar} as follows. First, by Theorem A thereof, either $\CR$ is linearly bounded or there is a definable binary operation $\cdot$ such that $(R,\le, +, \cdot)$ is a real closed field. Since $\mathcal Q\prec \CR$, in the latter case we would have a binary operation $\cdot_Q$ definable in $\mathcal Q$ making $(\Qq,\le, +, \cdot_Q)$ into a real closed field. But that is impossible, because $(\Qq, +)$ is the standard addition on $\Qq$, and therefore there exists at most one field structure (definable or not) expanding it, and $(\Qq,+,\cdot)$ is not real closed.
	
	So we are reduced to the linearly bounded case. By Theorem B of \cite{MilStar} every definable  endomorphism of $(R,+)$ is $\0$-definable. Thus, it will suffice to show that $x\mapsto \pi x$ is a definable endomorphism of $(R,+)$, since then it would be $\0$-definable, contradicting the assumption that $\mathcal Q\prec \CR$. But this should now be obvious, since $x\mapsto \pi x$ is definable as a function from $\mathcal Q$ to $\overline {\mathcal Q}$ and strongly continuous, so by Corollary \ref{unique} $x\mapsto \pi x$ is a definable continuous function in $\CR$, which is clearly an endomorphism of $(R,+)$.
\end{proof}

The above proof actually shows more:

\begin{corollary}
	If $\mathcal Q\equiv \Qq_{vs}^\pi$ then $\mathcal Q$ is not an o-minimal trace. In particular, $\Ralg^\pi$ is not an o-minimal trace. 
\end{corollary}
 
\begin{remark} We have shown that $\Ralg^\pi$ witnesses the fact that the class of o-minimal traces is not closed under reducts, and $\Qq_{vs}^\pi$ witnesses that the class of reducts of o-minimal traces is not closed under elementary equivalence. 
\end{remark}
 
\section{Atomic models and elimination of imaginaries}
In the o-minimal context definable Skolem functions have two main applications. The first, is a simple proof of the existence of atomic models, and in its stronger form of \emph{definable choice} it implies elimination of imaginaries. 

In the o-minimal context both properties can be proved under fairly general assumptions,  also for structures not supporting definable Skolem functions. In the present section we investigate these two properties in the weakly o-minimal non-valuational case. We show that all known examples of such structures do not have atomic models. We then discuss elimination of imaginaries, obtaining -- using the new machinery developed in \cite{EBY} -- some positive results.

The obstacle to the existence of atomic models is simple: if $\CM$ is strictly weakly o-minimal and non-valuational there exists a definable cut $C$ that is irrational over $\CM$,  i.e., $\CM$ is dense at $C$. But this need not be the case over arbitrary sets. In the case of o-minimal traces, our control over the definable closure operator, allows us to construct such examples: 

\begin{proposition}
	Let $\CM_0$ be a saturated o-minimal expansion of a group in the signature $\CL_0$. Let $\tp(c/\CM_0)$ be that of an irrational non-valuational cut, and $\CM$ the expansion of $\CM_0$ by the externally definable cut $x<c$. Then $\mathrm{Th}(\CM)$ does not have atomic models over arbitrary sets. 
\end{proposition}
\begin{proof}
	Let $A\subseteq M$ be a small set. We note that by weak o-minimality (and the group structure) the only types isolated over $A$ are algebraic. So it will suffice to find $A\subseteq M$ such that $\dcl_{\CM}(A)$ is not an elementary substructure. 
	
	Choose any small $\CN_0\subseteq \CM_0$ such that $\tp(c/N_0)$ is an irrational non-valuational cut. This can be done as follows: choose $\{a_i\}_{i\in \omega}\subseteq M$ inductively by $a_i\models \tp(c/A_{i-1})$, where $A_i=\dcl_{\CM_0}(A_{i-1}a_i)$ and $A_{-1}=\0$. Then set $\CN_0:=\bigcup_{i\in \omega}A_i$. Saturation of $\CM_0$ assures that this construction can be carried out. 
	
	Now let $a\in \CM$ be such that $\tp_{\CL_0}(a/N_0)=\tp_{\CL_0}(c/N_0)$. Then, by construction $\CN_0$ is dense in $\CN_0(a)$ and $a$ realises an irrational cut over $N_0$. It follows that $\tp(c/\CN_0a)$ is a rational cut. To show this it will suffice to prove that $|c-a|<b$ for all $b\in \CN_0(a)$. But because $N_0$ is dense in $\CN_0$ it will suffice to check the same thing for $b\in N_0$. This is now immediate from the choice of $a$ and the fact that $c$ realises an irrational  non-valuational cut over $\CN_0$. 
	
	It follows that $\CN_0(a)\not\equiv \CM$ (because $C$ -- the externally definable set $x<c$ -- is a rational cut over $\CN_0(a)$ but not over $M$). Finally, since, e.g., by \cite{vdDriesDense} $\CM$ is precisely the expansion of $\CM_0$ by cuts defined by $\CM_0(c)$ over $M_0$ we get that $\dcl_\CM(\CN_0(a))\subseteq \CN_0(ac)\cap M$. Since $\dim_{\CM_0}(c/M)=1$ it follows that $\dcl_\CM(\CN_0(a))=\CN_0(a)$, with the desired conclusion. 
\end{proof}

We point out that the key to the above proof is the fact that $\dcl_{\CM}(\CN_0(a))\subseteq \CN_0(ac)\cap M$. There are two observations that arise from this: 
\begin{enumerate}
	\item The exact same proof would work if $\CM$ were the expansion of $\CM_0$ by any number of distinct non-valuational irrational cuts. So it works for any o-minimal trace. 
	\item If the above proof is valid for a structure $\CM$ then it is valid for any ordered group reduct of $\CM$. 
\end{enumerate}
Thus we have shown: 
\begin{corollary}
	Let $\CM$ be an ordered group reduct of an o-minimal trace. Then $\mathrm{Th}(\CM)$ does not have atomic models over arbitrary sets. 
\end{corollary}

As we do not know of any example of a weakly o-minimal non-valuational structure whose theory is not that of an ordered group reduct of an o-minimal trace it is natural to ask: 

\begin{question}
	Assume that $T$ is a weakly o-minimal non-valuational theory expanding the theory of ordered group. Assume that $T$ has atomic models (i.e., over any set). Is $T$ o-minimal?
\end{question}

We now turn to the question of elimination of imaginaries. To the best of our knowledge, the only place the problem is addressed is \cite[Theorem 6.3]{Wenc1}. As our previous proposition shows, the assumptions in the main part of Wencel's result are not met in all known examples\footnote{Moreover, Wencel's proof seems to use \cite[Proposition 6.2]{Wenc1}. The justification for this usage is unclear to us.}. Here we take a different approach. To start we need: 
\begin{definition}
	Let $\CM$ be a weakly o-minimal expansion of an ordered group. Let $\overline M$ be the set of $\CM$-definable cuts, and $\CM_0^*$ the structure with universe $\overline M$ and whose atomic sets are $\cl_{\overline M}(D)$ for all $D$ $\CM$-definable over $\0$. 
\end{definition}

The following is the main result in the M.Sc. thesis of E. Bar Yehuda, \cite{EBY}: 
\begin{fact}\label{EBY}
	Let $\CM$ be a weakly o-minimal non-valuational expansion of an ordered group. Then $\CM_0^*$ is o-minimal and the structure it induces on $\CM$ is precisely $\CM$. Moreover, if $\CM\equiv \CN$ then $\CM_0^*\equiv \CN_0^*$. 
\end{fact}

Note that it follows from Corollary \ref{unique1} $\CM_0^*$ has the same definable sets as $\overline \CM$. We will also need the following fact from \cite{EBY}, which is key in the proof of Fact \ref{EBY}: 
\begin{fact}\label{NiceC}
	Let $\CM$ be a weakly o-minimal non-valuational expansion of an ordered group, $C\subseteq M^{n+m}$ any set $\CM$-definable over $\0$. Then there exists a $\tilde C$, $\CM_0^*$-definable over $\0$ such that for all $a\in M^n$ the set 
	\[\tilde C_a:=\{b\in \overline M^m: (a,b)\in \tilde C\}=\cl_{\overline M}(C_a)\] 
	where $C_a:=\{b\in  M^m: (a,b)\in C\}$. 
\end{fact}

We can now state the result: 
\begin{theorem}
	Let $\CM$ be a weakly o-minimal non-valuational expansion of an ordered group. Then $\CM_0^*$ eliminates imaginaries for $\CM$. I.e., for any equivalence relation $E$, $\0$-definable in $\CM$ there exists a function $f_E$,  $\0$-definable in $\CM_0^*$ such that for all $x,y\in \dom(E)$ we have $f(x)=f(y)$ if and only if $E(x,y)$. 
\end{theorem}
\begin{proof}
	Let $\overline E$ be as provided by Fact \ref{NiceC}. Since $\CM$ expands a group so does $\CM_0^*$. So $\CM_0^*$ has definable choice (possibly after naming one positive constant). So there is a $\0$-definable function $f:\dom(\overline E)\to \dom(\overline E)$ such that $f(x)=f(y)$ if and only if $\overline E_x=\overline E_y$. Now, if $x,y\in \dom(E)$ and $\models E(x,y)$ then $E_x=E_y$, implying that $\cl_{\overline M}E_x=\cl_{\overline M}E_y$, so that, by the construction of $\overline E$ and $f$ we get that $f(x)=f(y)$ So it remains to check that if $x,y\in \dom(E)$ and $\overline E_x=\overline E_y$ then $E(x,y)$. Towards that end, it will suffice to show that $E_x\cap E_y\neq \0$ (because $E$ is an equivalence relation). Since $\overline E_x=\cl_{\overline M}E_x$ it is enough to prove: \\
	
	\noindent\textbf{Claim:} If $C,D\subseteq M^n$ are such that $\cl_{\overline M}(C)=\cl_{\overline M}(D)$ then $C\cap D\neq \0$. \\
	
	\noindent Let $C,D$ be as in the claim. It is standard to check  that 
	\[\dim_{\overline \CM}(\cl_{\overline M}(C))=\dim_{\CM}(C)>\dim(\cl_M(C)\setminus C)\ge \dim_{\CM} (\partial C ).\tag{*}\] 
	So the assumptions imply that $\dim_\CM(C)=\dim_\CM(D)$. It will therefore be enough to prove the weaker claim that if 
	 \[
	 \dim_{\overline M}(\cl_{\overline M}(C)\cap \cl_{\overline M}(D))=\dim_{\CM}(C)= \dim_{\CM}(D)\tag{**}
	 \] then $C\cap D\neq \0$. 
	Fix strong cell decompositions $\mathcal C$ and $\mathcal D$ of $C$ and $D$ respectively. Since $\cl_{\overline M}(C)=\bigcup \{\cl_{\overline M}(C_i): C_i\in \mathcal C\}$, and similarly for $\CD$, there must be cells $C_i\in \CC$ and $D_j\in CD$ satisfying ($**$) above. Thus, we are reduced to proving the lemma under the assumption that $C,D$ are strong cells. 
	
	Now the proof of the claim for strong cells is an easy induction. If $C,D$ are $0$-cells, there is nothing to show.  If $C,D$ are 1-cells then $\cl_{\overline M}(C)$ and $\cl_{\overline M}(D)$ are closed intervals, and their intersection, by assumption, is also an interval. Since $M$ is dense in $\overline M$, we get that $C\cap D$ is an interval in $M$. 
	
	Essentially the same proof works if $C,D$ are open cells -- then the assumption implies that the intersection of their closures contains an open set, and by density of $M$ in $\overline M$ and ($*$) we get the desired conclusion. 
	
	In general, $C,D$ are graphs of definable functions. Let $\pi_C$ be a projection such that $\dim_{\CM}(\pi_C(C))=\dim(C)$. So, since $C$ is a cell, $\pi_C$ is injective on $C$ and $\pi_C(C)$ is an open cell. On a set of small co-dimension of $\cl_{\overline M}(C)$ the projection $\pi(C)$ is an injection so there exists a set of full dimension in $\cl_{\overline M}(C)\cap \cl_{\overline M}(D)$ such that $\pi_C$ is an injection. Restricting to the image of that set under $\pi_C$ the claim now follows by induction. 
\end{proof}

The above, implies in particular that if $\CM$ is a structure as above, and for any $\CM$-$\0$-definable function $f:M^n\to \overline M$ there exists an $\CM$-$\0$-definable function $\tilde f:M^n\to M^m$ (some $m$) such that $\tilde f (x)=\tilde f(y)$ if and only if $f(x)=f(y)$ then $\CM$ eliminates imaginaries. Indeed, if $E$ is an $\CM$-$\0$-definable equivalence relation on $M^n$ then by the theorem there exists a $\0$-definable (in $\CM_0^*$) function $ f_E:\overline M^n\to \overline M$ such that $f_E(x)=f_E(y)$ if and only if $E(x,y)$ for $x,y\in \dom(E)$. Since the structure induced on $M$ from $\CM_0^*$ is $\CM$, the restriction $f|M^n$ is $\CM$-definable. By \cite{EBY} it is $\CM$-$\0$-definable. So the function $\tilde f_E$ eliminates the imaginary $M^n/E$. 

In view of the above it seems natural to identify the pair $(\CM_0^*, \CM)$ with $\CM^{eq}$, suggesting that this may well be the right context for studying weakly o-minimal non-valuational structures.  
\bibliographystyle{plain}

\end{document}